\documentclass[a4paper,12pt]{article}
\usepackage{mathtools}
\usepackage{amsmath}
\usepackage{mathrsfs}
\usepackage{amssymb}
\usepackage{xypic}
\usepackage{amsthm}
\usepackage{dsfont}
\usepackage{graphicx}
\usepackage{bmpsize}
\usepackage{times}
\usepackage{ textcomp }
\usepackage{mathtools}
\usepackage{latexsym, empheq, fancybox}
\usepackage{mathrsfs}
\usepackage{exscale}
\usepackage{booktabs, array}
\usepackage[english]{babel}
\newtheorem{theorem}{Theorem}[section]

\newtheorem{definition}[theorem]{Definition}

\newtheorem{corollary}[theorem]{Corollary}

\newtheorem{claim}[theorem]{Claim}
\newtheorem{example}[theorem]{Example}
\newtheorem{proposition}[theorem]{Proposition}
\newtheorem{remark}[theorem]{Remark}
\usepackage{authblk}

\begin{document}

\title{Weighted  Mean Topological  Dimension
	\footnotetext {2010 Mathematics Subject Classification: 37B40, 37C45}}
\author{ Yunping Wang \\
	\small    School of Mathematical Sciences and Institute of Mathematics, Nanjing Normal University,\\
	\small   Nanjing 210046, Jiangsu, P.R.China\\
 \small    e-mail:  yunpingwangj@126.com 
}
\date{}

\maketitle

\begin{abstract}
This paper is devoted to the investigation of the weighted mean topological dimension in dynamical systems. We show that the weighted mean dimension  is not larger than the weighted metric mean dimension, which generalizes the classical result of  Lindenstrauss and Weiss \cite{LWE}. We also establish the relationship between the weighted mean dimension and the weighted topological entropy of dynamical systems, which indicates that each system with finite weighted topological entropy or small boundary property has zero weighted mean dimension.
\end{abstract}
\noindent
\textbf{Keywords:} weighted  mean topological dimension; weighted entropy; small boundary property

\section{Introduction}
One of the basic invariants of a dynamical system $(X,T)$ is its topological entropy. This quantifies to what extent nearby points diverge as the system evolves.  Adler, Konheim and McAndrew \cite{AKM} first introduced it as an invariant of topological conjugacy for studying dynamical systems in compact topological spaces. Subsequently, Dinaburg \cite{Din} and Bowen \cite{Bow2} presented a different, yet equivalent, definition in metric spaces.  Kenyon and Peres \cite{KE} studied the dimension of invariant sets and measures on the torus under expanding maps. In \cite{BF,F}, Barral and Feng  defined weighted topological pressure via relative thermodynamic formalism and subadditive thermodynamic formalism, in the particular case when the underlying dynamical systems $(X, T)$ and $(Y, S)$ are subshifts over finite alphabets.
Motivated by this work, Feng and Huang \cite{WH3} extended to general topological systems and introduced a  weighted version of entropy,  which provided a topological extension of dimension theory of invariant sets and measures on the torus under affine diagonal endomorphisms. For convenience, we recall the definition of the weighted entropy here.

Let $k\geq 2$. Assume that $(X_{i},d_{i}), i=1,\cdots,k$, are compact metric spaces, and $(X_{i},T_{i})$ are topological dynamical systems. Moreover, assume that for each $1\leq i\leq k-1,(X_{i+1},T_{i+1})$ is a factor of $(X_{i},T_{i})$  with a factor map $\pi_{i} : X_{i}\to X_{i+1};$ in other words, $\pi_{1},\cdots,\pi_{k-1}$ are continuous maps such that the following  diagrams commute,
\begin{equation*}
	\xymatrix{
	X_{1}\ar[r]^{\pi_{1}} \ar[d]_{T_{1}} &X_{2}\ar[r]^{\pi_{2}}\ar[d]_{T_{2}}  &\cdots \ar[r]^{\pi_{k-1}} &X_{k}\ar[d]_{T_{k}}\\
	X_{1}\ar[r]^{\pi_{1}}  &X_{2}\ar[r]^{\pi_{2}}  &\cdots \ar[r]^{\pi_{k-1}} &X_{k}
}
\end{equation*}
For convenience, we use $\pi_{0}$ to denote  the identity map on $X_{1}$. Define $\tau_{i}: X_{1}\to X_{i+1}$ by $\tau_{i}=\pi_{i}\circ\pi_{i-1}\circ\cdots\circ\pi_{0}$ for $i=0,1,\cdots,k-1.$
Let $M(X_{i},T_{i})$ denote the set of all $T_{i}$-invariant Borel probability measures on $X_{i}$ and $E(X_{i},T_{i})$ denote the set of ergodic measures. Fix
$\bf{a}$$=(a_{1},a_{2},\cdots,a_{k})\in \mathbb R^{k}$ with $a_{1}>0$ and $a_{i}\geq 0$ for $i\geq 2.$ For $\mu\in M(X_{1},T_{1})$, we call
  $$h^{\bf{a}}_{\mu}(T_{1}):=\sum_{i=1}^{k}a_{i}h_{\mu\circ\tau_{i-1}^{-1}}(T_{i})$$
  the $\bf{a}$-weighted measure-theoretic entropy of $\mu$ with respect to $T_{1}$, or simply, the $\bf{a}$-weighted entropy of $\mu$, where
$h_{\mu\circ\tau_{i-1}^{-1}}(T_{i})$ denotes the measure-theoretic entropy of $\mu\circ\tau_{i-1}^{-1}$ with respect to $T_{i}.$
\begin{definition}\label{define1.1}{\rm\cite{WH3}}
{\rm ($\bf{a}$-weighted Bowen ball).} For $x\in X_{1},n\in \mathbb N,\epsilon>0,$
let
\begin{align*}
B_{n}^{\bf{a}}(x,\epsilon):=\{y\in X_{1}: d_{i}(T_{i}^{j}\tau_{i-1}x, T_{i}^{j}\tau_{i-1}y)<\epsilon~~\text{for}~~~&1\leq i\leq k ~~\\&0\leq j\leq \lceil(a_{1}+\cdots+a_{i})n\rceil-1 \},
\end{align*}
where $\lceil u\rceil$ denotes the least integer $\geq u$. We call $B_{n}^{\bf{a}}(x,\epsilon)$ the $n$-th $\bf{a}$-weighted Bowen ball of radius $\epsilon$ centered at $x$.
\end{definition}

Return back to the metric spaces $(X_{i},d_{i})$ and topological dynamical systems $(X_{i},T_{i}),i=1,2,\cdots,k.$ For $n\in \mathbb N$, define a metric $d_{n}^{\bf a}$ on $X_{1}$ by
\begin{align*}
d_{n}^{\bf a}(x,y)=\sup\{d_{i}(T_{i}^{j}\tau_{i-1}x, T_{i}^{j}\tau_{i-1}y):  1\leq i \leq k, 0\leq j\leq \lceil(a_{1}+\cdots+a_{i})n\rceil-1\}.
\end{align*}
 \begin{definition}\label{sc}
	Let $n$ be a natural number, $\epsilon >0$, $\bf{a}$$=(a_{1},a_{2},\cdots,a_{k})\in \mathbb{R}^{k}$ and $K$ be a compact subset of $X_1$. A subset $F$ of $X_1$ is said to $({\bf a},n,\epsilon)$ spanning set of $K$ with respect to $T_1$ if for all $ x\in K $, there exists $ y\in F $ with $d_{n}^{\bf a}(x,y) \leq \epsilon$. Let $r_{n}^{\bf a}(\epsilon,K)$ denote the smallest cardinality of any $({\bf a},n,\epsilon)$ spanning set for $K$ with respect to $T_1$.
\end{definition}

\begin{definition}
	Let $n$ be a natural number, $\epsilon >0$, $\bf{a}$$=(a_{1},a_{2},\cdots,a_{k})\in \mathbb{R}^{k}$and $K$ be a compact subset of $X_1$. A subset E of $X_1$ is said to be $({\bf a},n,\epsilon)$ separated with respect to $T_1$ if $x , y\in E, x\neq y$, implies $ d_{n}^{\bf a}(x,y) > \epsilon$. Let $s_{n}^{\bf a}(\epsilon,K)$ denote the largest cardinality of any $({\bf a},n,\epsilon)$ separated set for $K$ with respect to $T_1$.
\end{definition}
The  ${\bf a}$-weighted topological entropy $ h^{\bf a}(X_1) $   using spanning sets and separated sets is defined  as follows.	
\begin{align*}
h^{\bf a}(X_1)&=\lim\limits_{\epsilon \rightarrow 0} \limsup\limits_{n\rightarrow\infty}\frac{1}{n}\log {r_n^{\bf a}}(X_1,\epsilon)\\&=\lim\limits_{\epsilon \rightarrow 0} \limsup\limits_{n\rightarrow\infty}\frac{1}{n}\log {s_n^{\bf a}}(X_1,\epsilon).
\end{align*}
Mean  dimension is a conjugacy invariant of topological dynamical systems which was introduced by  Gromov \cite{GRO}.  In 2000, Lindenstrauss and Weiss \cite{LWE} used it to answer  an open question raised by Auslander \cite{AU} that whether every minimal system $(X,T)$ can be imbedded in $[0, 1]^{\mathbb{Z}}$.  It turns out  that   mean dimension  is the right invariant to study  for the problem  of existence of an embedding into $(([0,1]^{D})^{\mathbb Z}, \sigma)$.  More applications and interesting relations have been found  \cite{GT, LT14, GLT16, GT19}. 
 The metric mean dimension was introduced in \cite{LWE} and they proved that metric mean dimension is  an upper bound of the  mean dimension. It allowed them to establish the relationship between the mean dimension and the topological
entropy of dynamical systems, which shows that each system with finite topological entropy has zero mean dimension.  This  invariant enables one to distinguish systems with infinite topological entropy. In \cite{LWE2, LWE}, they also introduced the small boundary property (SBP) by   replacing the empty set with the small set in the definition of zero dimensional space,  which could be seen as a dynamical version of being totally disconnected, and indicated  that a dynamical system with this property has  zero mean dimension. In order to  investigate the embedding  question  of weighted topological dynamical systems,
we will  introduce the concepts of  the weighted mean dimension and weighted metric mean dimension  inspired by the ideas of  \cite{ WH3, LWE}.

 In  this paper, we introduce the mean dimension for weighted version which include the classic case as $a_{1}=1, a_{i}=0$, $ i\geq2$. We also define the weighted metric mean dimension and show that topological dynamical system $(X_{1}, T_{1})$  with finite weighted topological entropy, or small boundary property has zero weighted mean  dimension. Meanwhile,
  we obtain an inequality between the weighted mean dimension and the weighted metric mean dimension. The following theorems present the main results of this paper.

\begin{theorem}\label{thm1}
  For any dynamical system $(X_1,T_1)$ and any metric $d_i$ compatible with the topology on $X_i,\ i=1,2,\cdots,k,$
  $$ {\rm mdim}^{\bf a}(X_1,T_1)\leq  {\rm mdim}_{M}^{\bf a}(X_1,d),$$
  where the definitions of these notions will be given in Section \ref{3}.
\end{theorem}

\begin{theorem}\label{thm3}
  If $(X_1,T_1)$ has the SBP, the weighted mean  dimension is zero.
\end{theorem}
The remainder of this paper is organized as follows. In Section \ref{2}, we recall  definitions and properties of the mean dimension. In Section \ref{3}, we define the weighted mean dimension and weighted metric mean dimension,  investigate the relationship between the two notions, and show that finite weighted topological entropy implies  the zero weighted mean dimension. In Section \ref{4}, we introduce the small  boundary property, and devote to proving that the SBP  yields the zero wighted mean dimension. In Section \ref{5}, we  present an example.

\section{Preliminaries}\label{2}
Next we recall some basic facts about mean dimension. These notions and terminologies can be found in \cite{Coo15} and \cite{LWE}. Let $X$ be a compact metric space and  $\alpha$ be a finite open cover of  $X$, we call that a cover $\beta$ refines $\alpha$$(\beta\succ\alpha)$, if every member of $\beta$ is a subset of some member of $\alpha$.
\begin{definition}
If $\alpha$ is an open cover of $X$, we shall denote
\begin{center}
$ord(\alpha)=\max\limits_{x\in X}\sum\limits_{U\in\alpha}1_{U}(x)\ \text{and}\  D(\alpha)=\min\limits_{\beta\succ\alpha}ord(\beta),$
\end{center}
where $\beta$ runs over all finite open covers of $X$ refining $\alpha$.
\end{definition}

\begin{definition}
Let $X$ be a topological space. The topological dimension $dim(X)$ of $X$ is the quantity defined by
\begin{center}
$dim(X)=\sup\limits_{\alpha}D(\alpha)$,
\end{center}
where $\alpha$ runs over all finite covers of $X$.
\end{definition}
\begin{definition}\label{md}
If $(X, T )$ is a dynamical system, then the mean dimension of $(X, T)$, denoted by $\text{mdim}(X, T)$, is defined by
$${\rm  mdim}(X,T)=\sup\limits_{\alpha}\lim\limits_{n\rightarrow \infty} \frac{1}{n} D(\bigvee\limits_{i=0}^{n-1}T^{-i}\alpha), $$
where $\alpha$ runs over all finite open covers of $X$.
\end{definition}
\begin{definition}
	Let $X $ and $Y$ be topological spaces. Let $\alpha$  be a finite open cover of $X$.  A continuous map $f: X \rightarrow Y$ will be called $\alpha$-compatible if it is possible to find a finite open cover $\beta$ of $f(X)$ such that $f^{-1}(\beta)\succ \alpha.$ We will use the notation $f \succ \alpha $ to denote that $f$ is $\alpha$-compatible.
\end{definition}
\begin{proposition}
Let $X$ and $Y$ be topological spaces and let $f:X\rightarrow Y$ be a continuous map. Let $\beta$ be a finite open cover of $Y$. Then one has
\begin{center}
$D(f^{-1}(\beta))\leq D(\beta)$.
\end{center}
\end{proposition}
\begin{proposition}\label{ad}
 Let $\alpha$ and $\beta$ be finite open covers of $X$, then
\begin{center}
$D(\alpha\vee\beta)\leq D(\alpha)+D(\beta)$.
\end{center}
\end{proposition}

\begin{proposition}\label{b}
	If $X$ is compact, $f: X \rightarrow Y$ a continuous function such that for  every $y \in Y$, $f^{-1}(y)$ is a subset of some $U \in \alpha$, then $f$ is $\alpha$-compatible.
\end{proposition}
\begin{proposition}\label{2.4}
If $\alpha$ is an open cover of $X$, then
$$D(\alpha)\leq k $$
iff there is an $\alpha$-compatible continuous function $f:X\rightarrow K$ where $K$ has topological dimension $k$.
\end{proposition}

\section{Proof of  Theorem \ref{thm1}}\label{3}
Firstly, we  introduce the concepts of  weighted mean dimension and weighted metric mean dimension.
Let $k\geq 2$. Assume that $(X_{i},d_{i}), i=1,2,\cdots,k$, are compact metric spaces, and $(X_{i},T_{i})$ are topological dynamical systems. Moreover, assume that for each $1\leq i\leq k-1,(X_{i+1},T_{i+1})$ is a factor of $(X_{i},T_{i})$  with a factor map $\pi_{i} : X_{i}\to X_{i+1}$. Fix ${\bf a}=(a_{1},\cdots,a_{k})\in \mathbb{R}^{k}$ with $a_{1}>0$ and $a_{i}\geq 0$ for $2\leq i \leq k$. Let $\alpha_i$ be a finite open cover of $(X_i,T_i)$ for  $ 1\leq i \leq k$. We define the weighted mean dimension as follows.
\begin{definition}\label{topological}
${\rm mdim}^{\bf a}(X_1,T_1)=\sup\limits_{\alpha_1,\cdot\cdot\cdot,\alpha_k}\limsup\limits_{n\rightarrow\infty} \dfrac{1}{n}{ D(\bigvee\limits_{i=1}^{ k}\bigvee\limits_{j=0}^{\lceil (a_{1}+a_{2}+....+a_{i})n\rceil-1 }T_{1}^{-j}\tau_{i-1}^{-1}\alpha_i)}$,
where $ \alpha_{i} $ is a finite open cover of $X_i,1\leq i \leq k$.
\end{definition}
\begin{remark}
 When $a_{1}=1, a_{i}=0$ $ (i\geq 2)$, the weighted mean dimension recovers the classic mean dimension. Our interest is on the general  case that $a_{i}\neq 0 $ $(1\leq i \leq k )$.
 It is worth pointing out that weighted mean dimension is  different from mean dimension in the case of $a_{i}\neq 0 $ $(2\leq i \leq k )$. In fact, we know the sequence
$\{D(\bigvee\limits_{i=0}^{n-1}T^{-i}\alpha)\}_{n\geq 1}$  is subadditive, however,  the sequence
 $$\{D(\bigvee\limits_{i=1}^{ k}\bigvee\limits_{j=0}^{\lceil (a_{1}+a_{2}+....+a_{i})n\rceil-1 }T_{1}^{-j}\tau_{i-1}^{-1}\alpha_i)\}_{n\geq 1}$$
 is not subadditive, then  the limit in the definition  \ref{topological} doesn't  exist in case of  $a_{i}\neq 0$ for some $2 \leq i \leq k.$
\end{remark}

\begin{proposition}
Let $X_1$ be a compact metric space and $T_1:X_1\rightarrow X_1$ a continuous map. Then one has ${\rm mdim}^{\bf a}(X_1,T_{1}^{n})=n\cdot\text{\rm mdim}^{\bf a}(X_1,T_1)$ for every positive integer $n$.
\end{proposition}
\begin{proof}
	For $1\leq i \leq k$, let $\alpha_{i}$ be a finite open cover of $X_{i}$. Set
	\begin{align*}
	w(\{{\alpha_i}\}_{i=1}^{ k},T_1,n)=\bigvee\limits_{i=1}^{ k}\bigvee\limits_{j=0}^{\lceil (a_{1}+a_{2}+....+a_{i})n\rceil-1 }T_{1}^{-j}\tau_{i-1}^{-1}\alpha_i.
	\end{align*}
    Since $\lceil (a_{1}+ \cdots + a_{i})mn \rceil-1\geq n(\lceil (a_{1}+\cdots + a_{i})m\rceil-1)$	for every integer $m\geq 1$,	we have the open cover
	\begin{align*}
	w(\{{\alpha_i}\}_{i=1}^{ k},T_1,mn)
	=\bigvee_{i=1}^{k}\bigvee\limits_{j=0}^{\lceil (a_1+a_2+....+a_i)(mn)\rceil-1}T_1^{-j}\tau_{i-1}^{-1}\alpha_i
	\end{align*}
	is finer than the open cover
	\begin{align*}
	w(\{{\alpha_i}\}_{i=1}^{k},T_1^n,m)=\bigvee\limits_{i=1}^{k}\bigvee\limits_{j=0}^{\lceil(a_{1}+a_{2}+....+a_{i})m\rceil-1 }(T_{1}^{n})^{-j}\tau_{i-1}^{-1}\alpha_i.
	\end{align*}	
    It follows that
	\begin{align*}
	D(\bigvee_{i=1}^{k}\bigvee\limits_{j=0}^{\lceil (a_1+a_2+....+a_i)(mn)\rceil-1}T_1^{-j}\tau_{i-1}^{-1}\alpha_i)\geq D(\bigvee\limits_{i=1}^{k}\bigvee\limits_{j=0}^{\lceil(a_{1}+a_{2}+....+a_{i})m\rceil-1 }(T_{1}^{n})^{-j}\tau_{i-1}^{-1}\alpha_i).
	\end{align*}
	We deduce that
	\begin{align}\label{n}
	n \cdot \text{mdim}^{\bf a}(X_1,T_1)\geq \text{mdim}^{\bf a}(X_1,T_1^n).
	\end{align}	
On the other hand, by $ \lceil (a_{1}+a_{2}+....+a_{i})nm\rceil-1\leq \lceil (a_{1}+a_{2}+....+a_{i})m\rceil n-1 $, we obtain
\begin{align*}
&\bigvee_{j=0}^{\lceil (a_{1}+a_{2}+....+a_{i})nm\rceil-1 }T_{1}^{-j}\tau_{i-1}^{-1}\alpha_i\prec \bigvee_{j=0}^{\lceil (a_{1}+a_{2}+....+a_{i})m\rceil n-1 }T_{1}^{-j}\tau_{i-1}^{-1}\alpha_i
\end{align*}
for each $1 \leq i \leq k$.
Let $\beta_{i}=\bigvee\limits_{j=0}^{n-1}T_i^{-j}\alpha_i$, then
\begin{align*}
w(\{{\alpha_i}\}_{i=1}^{ k},T_1,mn)\prec w(\{{\beta_i}\}_{i=1}^{ k},(T_1)^n,m).
\end{align*}
Moreover,
\begin{align*}
n \dfrac{1}{mn}D(\bigvee_{i=1}^{k}\bigvee\limits_{j=0}^{\lceil (a_1+a_2+....+a_i)(mn)\rceil-1}T_1^{-j}\tau_{i-1}^{-1}\alpha_i)\leq \dfrac{1}{m}D(\bigvee_{i=1}^{k}\bigvee\limits_{j=0}^{\lceil (a_1+a_2+....+a_i)m\rceil-1}(T_1^{n})^{-j}\beta_i).
\end{align*}
Then
\begin{align}\label{m}
n\cdot \text{mdim}^{\bf a}(X_1,T_1)\leq \text{mdim}^{\bf a}(X_1,T_1^n).
\end{align}
Inequalities ($\ref{n}$) and ($\ref{m}$) imply that
\begin{align*}
n\cdot \text{mdim}^{\bf a}(X_1,T_1)= \text{mdim}^{\bf a}(X_1,T_1^n).
\end{align*}

\end{proof}

\begin{proposition}
Let $\text {dim}(X_1)<\infty$. Then $\text{\rm mdim}^{\bf a}(X_1,T_1)=0$.
\end{proposition}
\begin{proof}
For every finite open cover $\alpha_i$ of $X_i$, $i=1,...,k$, we have
 \begin{align*}
 D(\bigvee_{i=1}^{k}\bigvee\limits_{j=0}^{\lceil (a_1+a_2+....+a_i)n\rceil-1}T_1^{-j}\tau_{i-1}^{-1}\alpha_i)\leq \text {dim}(X_1).
 \end{align*}
 As $\text {dim}(X_1)<\infty$, we deduce that
\begin{align*}
\limsup\limits_{n\rightarrow\infty}\dfrac{1}{n}D(\bigvee_{i=1}^{k}\bigvee\limits_{j=0}^{\lceil (a_1+a_2+....+a_i)n\rceil-1}T_1^{-j}\tau_{i-1}^{-1}\alpha_i)=0.
\end{align*}
Thus, we have that
\begin{align*}
\text{mdim}^{\bf a}(X_1,T_1)=0.
\end{align*}
\end{proof}

 We define a weighted version of metric mean dimension.
For an open cover $\alpha$, define the mesh of $\alpha$ according to the metric $d_{n}^{\bf a}$  by
\begin{center}
$mesh(\alpha ,d_{n}^{\bf a})=\max\limits_{U\in\alpha} diam(U)$.
\end{center}
Set
\begin{align}\label{s}
S^{{\bf a}}( X_1,\epsilon)=\limsup\limits_{n\rightarrow\infty}\inf\limits_{mesh(\alpha,d_{n}^{\bf a})<\epsilon} \frac{1}{n}\log|\alpha|.
\end{align}
It is obvious that
$S^{{\bf a}}(X_1,\epsilon )$ is monotone nondecreasing as $\epsilon\rightarrow 0$. 

\begin{definition}\label{metric}
We define the weighted metric mean dimension of $X_1$,
\begin{align}\label{S}
{\rm mdim}_{M}^{\bf a}(X_1,d)=\liminf \limits_{\epsilon\rightarrow 0} \frac{S^{{\bf a}}( X_1,\epsilon)}{|\log\epsilon|}.
\end{align}
\end{definition}
 Recall how the  weighted topological entropy is defined for a dynamical system $X_{1}$. According to Definition \ref{sc}, set
$$r_{n}^{\bf a}(X_1, \epsilon)=\min\{{|F|:F\subset X_1\ \text{is}\ ({\bf a},n,\epsilon)\text{-spanning set}}\},$$
\begin{equation*}
r^{\bf a}(X_1,\epsilon)= \limsup\limits_{n \rightarrow \infty}\frac{1}{n}\log r_{n}^{\bf a}(X_{1},\epsilon),
\end{equation*}
and
\begin{center}
$h^{\bf a}(X_1)=\lim\limits_{\epsilon\rightarrow 0}r^{\bf a}(X_1,\epsilon )$ .
\end{center}
 Notice that
\begin{center}
$ r^{\bf a}(X_1,\epsilon)\geq S^{{\bf a}}(X_1,2\epsilon)\geq r^{\bf a}(X_1,2\epsilon)$.
\end{center}
 Thus (\ref{S}) is equivalent to
\begin{center}
$ \text {mdim}^{\bf a}_{M}(X_1,d)=\liminf\limits_{\epsilon\rightarrow 0} \frac{r^{\bf a}(X_1,\epsilon)}{|\log\epsilon|}$.
\end{center}
\begin{proposition}
If ${\rm mdim}^{\bf a}_{M}(X_{1},d)\neq 0$, then $h^{\bf a}(X_{1})=\infty.$
\end{proposition}
\begin{corollary}
If $X_{1}$ has finite weighted topological entropy, then ${\rm mdim}^{\bf a}_{M}(X_{1},d)= 0$.
\end{corollary}

The next result is inspired in the Theorem 4.2 of \cite{LWE} and shows that the weighted metric mean dimension
is an upper bound for the weighted mean dimension.
\begin{proof}[Proof of Theorem \ref{thm1}]
  Assume that  for each $1 \leq i \leq  k-1$, $(X_{i+1}, T_{i+1})$ is a factor  of $(X_{i}, T_{i})$ with  a factor map $\pi_{i}: X_{i}\rightarrow X_{i+1}$. Define $\tau_{i}: X_{1} \rightarrow X_{i+1}$ by $\tau_{i}=\pi_{i} \circ \pi_{i-1}\circ \cdots \circ \pi_{0}$ for $i=0,1,\cdots, k-1$. As
  \begin{equation*}
  	\text{mdim}^{\bf a}(X_1,T_1)=\sup\limits_{\alpha_1,\cdots,\alpha_k}
  	\limsup\limits_{n\to\infty}\dfrac{1}{n}D(\bigvee\limits_{i=1}^k\bigvee\limits_{j=0}^{\lceil(a_1+a_2+\cdots+a_i)n\rceil-1}
  	T_1^{-j}\tau_{i-1}^{-1}\alpha_i),
  \end{equation*}
  it is sufficient to show that for any finite open cover $\alpha_i$ of $X_i,\ i=1,2,\cdots,k$,
  $$\limsup\limits_{n\to\infty}\dfrac{1}{n}D(\bigvee\limits_{i=1}^k\bigvee\limits_{j=0}^{\lceil(a_1+a_2+\cdots+a_i)n\rceil-1}
  T_1^{-j}\tau_{i-1}^{-1}\alpha_i)\leq \text {mdim}_{M}^{\bf a}(X_1,d).$$
Since $\alpha_i$ be a finite open cover of $X_i$, then $\tau_{i-1}^{-1}\alpha_i$ is a finite open cover of $X_1$.
We can refine $\tau_{i-1}^{-1}\alpha_i$ to be of the form
$$\tau_{i-1}^{-1}\alpha_i=\{U_{i_1},V_{i_1}\}\vee\{U_{i_2},V_{i_2}\}\vee\cdots\vee\{U_{i_{r_i}},V_{i_{r_i}}\},$$
where for $1\leq j \leq r_{i}$, $1\leq i \leq k$, $\{U_{i_j},V_{i_j}\}$ is an open cover of $X_1$ . Set $w_{i,j}:\ X_1\rightarrow[0,1]$ by
\begin{align}
  w_{i,j}(x)=\dfrac{d(x,X_1\backslash V_{i_j})}{d(x,X_1\backslash V_{i_j})+d(x,X_1\backslash U_{i_j})},\ 1\leq i \leq k,~1\leq j  \leq r_i.
\end{align}
Then $w_{i,j}$ is Lipschitz, and
\begin{align}\label{k}
\begin{split}
 U_{i_j} & =w_{i,j}^{-1}[0,1) \\
V_{i_j} & =w_{i,j}^{-1}(0,1] .
\end{split}
\end{align}
Let $C_L$ be a bound on the Lipschitz constants of all $w_{i,j}$.

For any $N$ define $F(N,\cdot):\ X_1\rightarrow
[0,1]^{\sum\limits_{i=1}^{k}r_i\lceil(a_1+a_2+\cdots+a_i)N\rceil}$ by
\begin{align*}
  F(N,x)=( & w_{1,1}(x),\cdots,w_{1,r_1}(x), \\
   & w_{1,1}(T_{1}x),\cdots,w_{1,r_{1}}(T_1x),\cdots \\
   & w_{1,1}(T_1^{\lceil a_1N\rceil-1}x),\cdots,w_{1,r_1}(T_1^{\lceil a_1N\rceil-1}x) \\
   & \cdots \\
   & w_{k,1}(x),\cdots,w_{k,r_k}(x), \\
   & w_{k,1}(T_1x),\cdots,w_{k,r_k}(T_1x), \\
   & w_{k,1}(T_1^{\lceil(a_1+a_2+\cdots+a_k)N\rceil-1}x),\cdots,w_{k,r_k}(T_1^{\lceil(a_1+a_2+\cdots+a_k)N\rceil-1}x)).
\end{align*}
From  Proposition \ref{b} and (\ref{k}) we see that
$F(N,\cdot)\succ\bigvee\limits_{i=1}^{k}\bigvee\limits_{j=0}^{{\lceil(a_1+a_2+\cdots+a_i)N\rceil-1}}T_1^{-j}
\tau_{i-1}^{-1}\alpha_i$. Usually, if
$$S\subset\{1,\cdots,\sum\limits_{i=1}^{k}r_i{\lceil(a_1+a_2+\cdots+a_i)N\rceil}\},$$ then $F(N,x)_S\in[0,1]^{|S|}$
is the projection of $F(N,x)$ to the coordinates in the index set $S$. 
\begin{claim}\label{1.5.13}
  Let $\epsilon>0$, $\ D= {\rm mdim}_{M}^{\bf a}(X_1,d)$. If $N$ is larger than some $N(\epsilon)$, there exists
  $\xi\in(0,1)^{\sum\limits_{i=1}^{k}r_i{\lceil(a_1+a_2+\cdots+a_i)N\rceil}}$ such that for any $|S|\geq
  (D+\epsilon)N$,$$\xi_S\notin F(N,X)_S.$$
\end{claim}
\begin{proof}[Proof of Claim \ref{1.5.13}]
  Let $$\delta<(2^{\sum\limits_{i=1}^{k}r_i{\lceil(a_1+a_2+\cdots+a_i)\rceil}}(2C_L)^{2D})^{-\frac{2}{\epsilon}},$$
   such that $$\dfrac{S(X_1,\epsilon,d_n^{\bf a})}{|\log\delta|}\leq \text {mdim}_{M}^{\bf a}(X_1,d)+\epsilon/4.$$
  Fix $N$  large enough (depending on $\epsilon$ and $\delta$), $X_1$ can be covered by
  $\delta^{-(D+\epsilon/2)N}$ weighted Bowen balls of the form
  \begin{align*}
 B_{N}^{\bf a}(x,\delta)=&\{x'\in X_1:d_i(T_i^j\tau_{i-1}x,T_i^j\tau_{i-1}x')<\delta, \
 \ \forall1 \leq i\leq k \ \text{and}\ \\& 0\leq j\leq\lceil(a_1+a_2+\cdots+a_i)N\rceil-1\}.
  \end{align*}
   Notice that $C_L$ is a Lipschitz constant for all $w_{i,j}$, we have
  $$F(N,B_N^{\bf a}(x,\delta))\subset\{a\in[0,1]^{\sum\limits_{i=1}^{k}r_i{\lceil(a_1+a_2+\cdots+a_i)N\rceil}}:
  \|F(N,x)-a\|_{\infty}\leq C_L\delta\}.$$
  Then the set $F(N, X_1)$ can be covered by $\delta^{-N(D+\epsilon/2)}$ balls in the $\parallel  \cdot \parallel_{\infty}$ norm of radius $C_{L}\delta.$
  Enumerate these weighted Bowen balls as $B^{\bf a}(k),\ k=1,\cdots,K=\delta^{-(D+\epsilon/2)N}.$
  Choose $\xi$ with uniform probability in $[0,1]^{\sum\limits_{i=1}^{k}r_i{\lceil(a_1+a_2+\cdots+a_i)N\rceil}}$.

  Then for a set
  $S\subset\{1,\cdots,\sum\limits_{i=1}^{k}r_i{\lceil(a_1+a_2+\cdots+a_i)N\rceil}\},$
  $$P(\xi_S\in F(N,X_1)_S)\leq\sum\limits_{k=1}^KP(\xi_S\in B^{\bf a}(k)|_S)\leq\delta^{-(D+\epsilon/2)N}
  (2C_L\delta)^{|S|}.$$
  Hence
  \begin{align*}
    P(& \exists S:|S|\geq(D+\epsilon)N\ \text{and}\ \xi_S\in F(N,X_1)_S)   \\
     & \leq\sum\limits_{|S|\geq(D+\epsilon)N}P(\xi_S\in F(N,X_1)_S) \\
     & \leq(\sharp\ \text{of such }S)\times\delta^{-(D+\epsilon/2)N}(2C_L\delta)^{(D+\epsilon)N} \\
     & \leq 2^{\sum\limits_{i=1}^{k}r_i{\lceil(a_1+a_2+\cdots+a_i)N\rceil}}
     \times\delta^{-(D+\epsilon/2)N}(2C_L\delta)^{(D+\epsilon)N}\\
     & \leq 2^{\sum\limits_{i=1}^{k}r_i{\lceil(a_1+a_2+\cdots+a_i)N\rceil}}\times
     ((2C_L)^{2D}\delta^{\epsilon/2})^N\ll1,
  \end{align*}
  and so, with high probability, a random $\xi$ will satisfy the requirements.
\end{proof}
\begin{claim}\label{1.5.4}
If $\pi: F(N,X_1)\rightarrow [0,1]^{\sum\limits_{i=1}^{k}r_i{\lceil(a_1+a_2+\cdots+a_i)N\rceil}}$ satisfies
  \begin{align*}
    &\{1\leq k\leq\sum\limits_{i=1}^{k}r_i{\lceil(a_1+a_2+\cdots+a_i)N\rceil}:\xi_k=a\}\subset \\
     &\{1\leq k\leq\sum\limits_{i=1}^{k}r_i{\lceil(a_1+a_2+\cdots+a_i)N\rceil}:\pi(\xi)|_k=a\},
  \end{align*}
   for both $a=0$ and $1$, and all
  $\xi\in[0,1]^{\sum\limits_{i=1}^{k}r_i{\lceil(a_1+a_2+\cdots+a_i)N\rceil}}.$
  Then $\pi\circ F(N,X_1)$ is compatible with
  $\bigvee\limits_{i=1}^k\bigvee\limits_{j=0}^{\lceil(a_1+a_2+\cdots+a_i)N\rceil-1}T_1^{-j}\tau_{i-1}^{-1}
  \alpha_i$.
\end{claim}
\begin{proof}[Proof of Claim \ref{1.5.4}]
  In fact, if $\xi\in[0,1]^{\sum\limits_{i=1}^{k}r_i{\lceil(a_1+a_2+\cdots+a_i)N\rceil}},$
  define for $0\leq m_i< \lceil(a_1+a_2+\cdots+a_i)N\rceil$ and $1\leq n_i<r_i$
  $$W_{m_i,n_i}=
\begin{cases}
  T_1^{-m_i}U_{n_i} &   \text{if }\xi_{m_ir_i+n_i}=0, \\
  T_1^{-m_i}V_{n_i} &   \text{otherwise.} \\
\end{cases}
$$

It is easy to see  that
$$F^{-1}\pi^{-1}(\xi)\subset\bigcap\limits_{i=1}^k\bigcap\limits_{0\leq m_i<\lceil(a_1+a_2+\cdots+a_i)N\rceil,
\atop 1\leq n_i< r_i}
W_{m_i,n_i}\in\bigvee\limits_{i=1}^k\bigvee\limits_{j=0}^{\lceil(a_1+a_2+\cdots+a_i)N\rceil-1}
T_1^{-j}\tau_{i-1}^{-1}\alpha_i.$$
\end{proof}

Let $\epsilon>0$. Choose $\overline{\xi}, N$ as in Claim $\ref{1.5.13}$. Set
$$\Phi=\{\xi\in[0,1]^{\sum\limits_{i=1}^{k}r_i{\lceil(a_1+a_2+\cdots+a_i)N\rceil}}:
\xi_k=\overline{\xi}_k\ \text{for more than }(D+\epsilon)N\ \text{indexes }k\};$$
thus $F(N,X_1)\subset[0,1]^{\sum\limits_{i=1}^{k}r_i{\lceil(a_1+a_2+\cdots+a_i)N\rceil}}\backslash\Phi$.
 We  denote  the latter set briefly by  $\Phi^C$. We shall construct a continuous retraction $\pi$ of
$\Phi^c$ onto the $\lfloor(D+\epsilon)N\rfloor+1$ skeleton of the cube $I=
[0,1]^{\sum\limits_{i=1}^{k}r_i{\lceil(a_1+a_2+\cdots+a_i)N\rceil}}$. For $1\leq m \leq \sum\limits_{i=1}^{k}r_i{\lceil(a_1+a_2+\cdots+a_i)N\rceil} $,  define
$$J_m=\{\xi\in I:\xi_i\in\{0,1\}\ \text{for at least }m\ \text{indexes }1\leq i\leq
 \sum\limits_{i=1}^{k}r_i{\lceil(a_1+a_2+\cdots+a_i)N\rceil}\}.$$

Since $\overline{\xi}$ is in the  interior of $I$, one can define $\pi_1:I\backslash\{\overline{\xi}\}
\rightarrow J_1$ by mapping each $\xi$ to the intersection of the ray starting at $\overline{\xi}$ and passing
through $\xi$ and $J_1$.
For each of the $(\sum\limits_{i=1}^{k}r_i{\lceil(a_1+a_2+\cdots+a_i)N\rceil}-1)$-dimensional cubes $I^l$ that
comprise $J_1$, we can define a retraction on $I^l$ in a similar fashion using as a center the projection of
$\overline{\xi}$ onto  $I^{l}$. This will define a continuous retraction $\pi_{2}$ of $\Phi^{C}$ into $J_{2}$. As long as there is some intersection of $\Phi$  with the cubes in $J_{m}$  this process can be continued, thus we get finally a continuous projection  $\pi$ of $\Phi^{C}$ onto $J_{m_{0}}$, with
$$m_0+\lfloor(D+\epsilon)N\rfloor+1=\sum\limits_{i=1}^{k}r_i{\lceil(a_1+a_2+\cdots+a_i)N\rceil}.$$
Obviously, $\pi$ satisfies the conditions of Claim $\ref{1.5.4}$. Therefore
$$\pi\circ F(N,\cdot)\succ
\bigvee\limits_{i=1}^k\bigvee\limits_{j=0}^{\lceil(a_1+a_2+\cdots+a_i)N\rceil-1}T_1^{-j}\tau_{i-1}^{-1}\alpha_i.$$
Moreover, as $F(N,\cdot)\subset\Phi^C$, $$\pi\circ F(N,\cdot)\subset J_{m_0},$$ the latter having topological dimension
$\lfloor(D+\epsilon)N\rfloor+1$.

In conclusion, we have constructed a
$\bigvee\limits_{i=1}^k\bigvee\limits_{j=0}^{\lceil(a_1+a_2+\cdots+a_i)N\rceil-1}T_1^{-j}\tau_{i-1}^{-1}\alpha_i$
compatible function from $X_1$ to a space of topological dimension $\leq$$(D+\epsilon)N+1$, and according to Proposition $\ref{2.4}$,
$$D(\bigvee\limits_{i=1}^k\bigvee\limits_{j=0}^{\lceil(a_1+a_2+\cdots+a_i)N\rceil-1}T_1^{-j}\tau_{i-1}^{-1}\alpha_i)
\leq(D+\epsilon)N+1.$$
As $\epsilon>0$ can be chosen as small as we like, $\alpha_{i}$ as defined as we like for $1\leq i \leq k$, we conclude that $$\text {mdim}^{\bf a}(X_1, T_{1})\leq D.$$
\end{proof}

\section{Proof of Theorem \ref{thm3} }\label{4}
In this section, we discuss the relationship between the small-boundary property and weighted mean  dimension.
\begin{definition}
  Let $(X, T)$ be a dynamical system, and $E$  a subset of $X$. We define the { orbit capacity} of the set $E$ to be $$ocap(E)=\lim\limits_{n\to\infty}\sup\limits_{x\in X}\dfrac{1}{n}\sum_{i=0}^{n-1}1_E(T^ix).$$
  A set $E$ will be called { small} if $ocap(E)=0$.
\end{definition}
\begin{definition}
  A dynamical system $(X,T)$ has the { small boundary property} (SBP) if every point $x\in X$ and every open set $U\ni x$ there is a neighborhood $V\subset U$ of $x$ with small boundary.
\end{definition}
\begin{proposition}\rm\cite{LWE}\label{smal}
  If $(X,T)$ has the SBP, then for every open cover $\alpha$ of $X$ and every $\epsilon$ there is a subordinate partition of unity $\phi_j:\ X\rightarrow [0,1],\ (j=1,\cdots,|\alpha|)$ such that
  \begin{align}
    &\sum_{j=1}^{|\alpha|}\phi_j(x) = 1, \notag\\
    &supp(\phi_j) \subset U\ \text{for some}\ U\in\alpha,\ j=1,\cdots,|\alpha|, \label{10} \\
    &ocap(\bigcup\limits_{j=1}^{|\alpha|}\phi_j^{-1}(0,1))< \epsilon. \notag
  \end{align}
\end{proposition}

\begin{proof}[Proof of Theorem \ref{thm3}]
 Assume that  for each $1 \leq i \leq  k-1$, $(X_{i+1}, T_{i+1})$ is a factor  of $(X_{i}, T_{i})$ with  a factor map $\pi_{i}: X_{i}\rightarrow X_{i+1}$. Define $\tau_{i}: X_{1} \rightarrow X_{i+1}$ by $\tau_{i}=\pi_{i} \circ \pi_{i-1}\circ \cdots \circ \pi_{0}$ for $i=0,1,\cdots, k-1$. Let $(X_1,T_1)$ have the SBP and $\alpha_i$ is a finite open cover of $X_i$ for $1 \leq  i \leq k$.
  Given $\epsilon>0$.  According to the  Proposition $\ref{smal}$,  we can construct a partition of unity subordinate to  $\beta_i=\tau_{i-1}^{-1}\alpha_i$  such  that  satisfies  (\ref{10}) for $i=1,\cdots, k$. Let $|\beta_i|=c_i$ and $A=\bigcup\limits_{m=1}^{c_i}\phi_{m,i}^{-1}(0,1)$ for each $1 \leq i \leq k$.
 Let $N$ is large enough such that
  \begin{align}\label{12}
    \dfrac{1}{\lceil(a_1+a_2+\cdots+a_i)N\rceil}\sum\limits_{j=0}^{\lceil(a_1+a_2+\cdots+a_i)N\rceil-1}1_A(T_{1}^j(x))<\epsilon,\ \forall x\in X_1,i=1,\cdots ,k.
  \end{align}
  Let $\Phi_{i}:(X_1,T_1)\rightarrow\mathbb{R}^{c_i}$ is given by $$x\rightarrow(\phi_{1,i}(x),\cdots,\phi_{c_i,i}(x)).$$
Given  $1\leq i \leq k $,  define the map $f_N^{i} :X_1\rightarrow\mathbb{R}^{c_i\lceil(a_1+a_2+\cdots+a_i)N\rceil}$ by $$f_N^{i}(x)=(\Phi(x),\Phi(T_{1}x),\cdots,\Phi(T_{1}^{\lceil(a_1+a_2+\cdots+a_i)N\rceil-1}x)).$$
  We claim that $f_N^{i}(X_1)$ is a subset of a finite number of $\epsilon c_i\lceil(a_1+a_2+\cdots+a_i)N\rceil$ dimensional affine subspaces of $\mathbb{R}^{c_i\lceil(a_1+a_2+\cdots+a_i)N\rceil}$.

  Indeed, let $e_l^m,\ l=1,\cdots,\lceil(a_1+a_2+\cdots+a_i)N\rceil,\ m=1,\cdots,c_i$ be the standard base of $\mathbb{R}^{c_i\lceil(a_1+a_2+\cdots+a_i)N\rceil}$.
  Define for every $I=\{i_1,i_2,\cdots,i_{N'}\},\ N'<\epsilon\lceil(a_1+a_2+\cdots+a_i)N\rceil$, and every $\xi\in\{0,1\}^{c_i\lceil(a_1+a_2+\cdots+a_i)N\rceil}$
  $$C(I,\xi,i)=\text {span}\{e_l^m:\ l\in I,\ 1\leq m\leq c_i\}+\xi.$$
  Then by (\ref{12}),
  $$f_N^{i}(X_1)\subset\bigcup\limits_{|I|<\epsilon\lceil(a_1+a_2+\cdots+a_i)N\rceil,\ \xi}C(I,\xi,i).$$
  It is obvious that $f_N^{i}$ is $\bigvee\limits_{j=0}^{\lceil(a_1+a_2+\cdots+a_i)N\rceil-1}T_{1}^{-j}\beta_i$-compatible.
 According to Proposition $\ref{2.4}$, we have
  $$D(\bigvee\limits_{j=0}^{\lceil(a_1+a_2+\cdots+a_i)N\rceil-1}T_{1}^{-j}\beta_i)<\epsilon c_i\lceil(a_1+a_2+\cdots+a_i)N\rceil.$$
  By subadditive of $D(\alpha)$, we thus get
  \begin{align*}
    D(\bigvee\limits_{i=1}^k\bigvee\limits_{j=0}^{\lceil(a_1+a_2+\cdots+a_i)N\rceil-1}T_1^{-j}\beta_i) & \leq \sum\limits_{i=1}^{k}D(\bigvee\limits_{j=0}^{\lceil(a_1+a_2+\cdots+a_i)N\rceil-1}T_1^{-j}\beta_i)\\
     & \leq \sum\limits_{i=1}^{k}\epsilon c_i \lceil(a_1+a_2+\cdots+a_i)N\rceil.
  \end{align*}
  Since $\sum\limits_{i=1}^{k}c_i<\infty$, it follows that $(X_1,T_1)$ has zero weighted mean dimension.
\end{proof}
\section{Example}\label{5}

\begin{example}
Put ${\bf a}=\left(  1,1\right)   $.  Assume that $X_{1}=([0,1]^{2})^{\mathbb{Z}}$ and $X_{2}=[0,1]^{\mathbb{Z}}$. Define $T_{1}: X_{1} \rightarrow X_{1}$ by  $T_{1}((x_{m})_{m\in \mathbb{Z}})=(x_{m+1})_{m\in \mathbb{Z}}, $  and  define $T_{2}: X_{2} \rightarrow X_{2}$ by $T_{2}((a_{m})_{m\in \mathbb{Z}})=(a_{m+1})_{m\in \mathbb{Z}}$.
 Let $\tau_{1}$ be the canonical projection from $X_{1}$ to $X_{2}$, that is to say, $\tau_{1}((x_{m}^{1}, x_{m}^{2})_{m \in  \mathbb{Z}})= (x_{m}^{1})_{m\in \mathbb{Z}}$. Then $(X_{2}, T_{2})$ is the factor of $(X_{1}, T_{1})$ associated with the factor map $\tau_{1}$. We define  a distance $d_{1}$ on $X_{1}$  by
$$d_{1}(x,y)= \sum\limits_{m\in \mathbb{Z}} 2^{-|m|}d'(x_{m},y_{m}), \ \ \ \ (x=(x_{m})_{m \in \mathbb{Z}}, y=(y_{m})_{m\in \mathbb{Z}} \in X_{1} ),$$
 where $d^{'}(s, t)= \max\{|s_{1}-t_{1}|, |s_{2}- t_{2}|\}$ is a distance on $[0,1]^{2}$.
 We define a distance $d_{2}$ on $X_{2}$ by
 $$ d_{2}(a, b)=\sum\limits_{m\in \mathbb{Z}} 2^{-|m|}|a_{m}-b_{m}|, \ \ \ \  (a=(a_{m})_{m\in \mathbb{Z}}, b=(b_{m})_{m\in \mathbb{Z}} \in X_{2}).$$
 We calculate the weighted metric mean dimension of $X_{1}$. Let $\epsilon >0$ and set  $l=\lfloor \log_{2}(\frac{4}{\epsilon})\rfloor$. Then $\sum\limits_{|n| \leq l} 2^{-|n|} \leq \frac{\epsilon}{2}$.
 We consider   an open covering of $[0,1]^{2}$ by
 $$ D_{\frac{\epsilon}{6}}(k,k^{'})=(\frac{(k-1)\epsilon}{12}, \frac{(k+1)\epsilon}{12}) \times (\frac{(k^{'}-1)\epsilon}{12}, \frac{(k^{'}+1)\epsilon}{12}),~~1 \leq k,k^{'} \leq \lfloor\frac{12}{\epsilon}\rfloor.$$
 For $n\geq 1$, consider
 $$  ([0,1]^{2})^{\mathbb{Z}}= \bigcup\limits_{1\leq k_{-l}, \cdots ,k_{2n+l}\leq \frac{12}{\epsilon}\atop 1\leq k^{'}_{-l}, \cdots ,k^{'}_{2n+l}\leq \frac{12}{\epsilon} } \{x| x_{-l}\in D_{\frac{\epsilon}{6}}(k_{-l}, k^{'}_{-l}),\cdots, x_{2n+l}\in D_{\frac{\epsilon}{6}}(k_{2n+l}, k^{'}_{2n+l})\}.$$
  We define $\sharp (X, d, \epsilon)$ as the minimum cardinality $N$ of the open covering $\left\lbrace U_{1}, \cdots, U_{N}\right\rbrace$ of $X$ such that all $U_{n}$ have diameter smaller than $\epsilon$.
 Each open  set in the right hand side has diameter less than $\epsilon$ with respect to the distance $d_{n}^{\bf a}$. Hence
 $$\sharp([0,1]^{\mathbb{Z}}, d_{n}^{\bf a}, \epsilon)\leq (1+\lfloor\frac{12}{\epsilon}\rfloor)^{4n+4l+2}= (1+\lfloor\frac{12}{\epsilon}\rfloor)^{4n+4\lfloor \log_{2}(\frac{4}{\epsilon})\rfloor+2}.$$ 
 On the other hand, any two distinct points in the sets
 $$ \{ x \in ([0,1]^{2})^{\mathbb{Z}}| x_{m}\in (k_{1}, k_{2} )  \  \text{for} \  \text{all} \ 0\leq m < 2n\}  $$
 where $k_{i}\in \{ 0, \epsilon , 2\epsilon, \cdots, \lfloor\frac{1}{\epsilon}\rfloor \epsilon\},$ have distance $\geq \epsilon$ with respect to $d_{n}^{\bf a}.$ It follows that $\sharp ([0,1]^{\mathbb{Z}}, \epsilon) \geq (1+\lfloor \frac{1}{\epsilon}\rfloor)^{4n}.$ Therefore
 $$ S(X_{1}, \epsilon)= \limsup\limits_{n\rightarrow \infty } \frac{1}{n}\log \sharp (X_{1}, d_{n}^{\bf a}, \epsilon)\sim 4|\log\epsilon| \ \ \ (\epsilon \rightarrow 0).$$
 Thus ${\rm mdim}_{M}^{\bf a}(X_{1}, d)={\rm mdim}_{M}^{\bf (1,1)}(X_{1}, d)=4$.
\end{example}

\end{document}